\documentclass[amsmath,amssymb,aps,floatfix,nofootinbib,dvipsnames]{revtex4} 
\usepackage{bm}
\usepackage{color}
\usepackage{dcolumn}
\usepackage{graphicx}
\usepackage{hyperref}	
\usepackage{indentfirst}
\usepackage{mathptmx}
\usepackage{amsthm}
\newtheorem{theorem}{Theorem}           
\newtheorem{lemma}{Lemma}               

\theoremstyle{definition}
\newtheorem{definition}{Definition}

\begin{document}


\title{A New Method to Solve the Gaussian Integral: A Solution Inspired by Taylor Series}

\author{L\'{a}zaro L. Sales$^{1,}$\footnote{lazarosales@alu.uern.br}}

\author{Jonatas A. Silva$^{2,}$\footnote{arizilanio@gmail.com}}

\author{Eli\^{a}ngela P. Bento$^{3,}$\footnote{eliangela.pb@gmail.com}}

\author{Hidalyn T. C. M. Souza$^{4,}$\footnote{hidalyn.souza@ufersa.edu.br}}

\author{Antonio D. S. Farias$^{4,}$\footnote{antonio.diego@ufersa.edu.br}}

\author{Ot\'{a}vio P. Lavor$^{4,}$\footnote{otavio.lavor@ufersa.edu.br}}

\affiliation{$^1$Departamento de F\'{\i}sica, Universidade do Estado do Rio Grande do Norte, 59610-210, Mossor\'o-RN, Brazil}

\affiliation{$^2$Departamento de Ci\^{e}ncias e Tecnologias, Universidade Federal Rural do Semi-\'{A}rido, 59780-000, Cara\'{u}bas-RN, Brazil}

\affiliation{$^3$Departamento de F\'{i}sica, Universidade Federal da Para\'{i}ba, 58059-970, Jo\~{a}o Pessoa-PB, Brazil}

\affiliation{$^4$Departamento de Ci\^{e}ncias Exatas e Naturais, Universidade Federal Rural do Semi-\'{A}rido, 59900-000, Pau dos Ferros-RN, Brazil}

\date{\today}

\begin{abstract}
	
	In this paper, we have proposed a new method for solving the Gaussian integral. Introducing a parameter that depends on a $n$ index, we have found a general solution for this type of integral inspired by Taylor series of a simple function. We have demonstrated that this parameter represents the Taylor series coefficients of this function, a result very newsworthy. We have also introduced some Theorems that are proved by mathematical induction. The proposed method in this work has shown more practical and accessible than some methods found in the literature. As a test for the method, we have investigated a non-extensive version for the particle number density in Tsallis framework, which enabled us to evaluate the functionality of the method. Besides, solutions for a certain class of the gamma and factorial functions are derived. Moreover, we have presented a simple application in fractional calculus. In conclusion, we believe in the relevance of this work because it presents a new form of solving the Gaussian integral having the differential calculus as a tool.    
	\\
	

\end{abstract}

\maketitle



\section{Introduction} 

Gaussian integral also known as probability integral is the integral of the function $\exp(-x^{2})$ over the entire line ($-\infty, \infty$). Solutions of this type of integral involve the so-called gamma functions introduced by Euler in 18th century and improved by Legendre, Gauss and Weierstrass \cite{Davis,Gronau}. The Gaussian integral has a wide range of applications in several areas of knowledge. Indeed, when we do a slight change of variables it is possible to compute the normalizing constant of the normal distribution in probability and statistics \cite{Stahl, Murray}. In physics the Gaussian integral appears frequently in quantum mechanics \cite{QGreiner}, to find the probability density of the ground state of the harmonic oscillator, in the path integral formulation \cite {Sakurai}, to find the propagator of the harmonic oscillator and in statistical mechanics \cite{Greiner, Pathria, Salinas}, to find its partition function. 

First we recall the solution of following Gaussian integral, as bellow:
\begin{equation} 
\int_{-\infty}^{+\infty}e^{-\alpha x^2}\;dx = \sqrt{\frac{\pi}{\alpha}}\;,
\end{equation}
where $\alpha \in \mathbb{R},\; \alpha \neq 0$. There are several methods to solve this type of integral, the most well known and widely used in textbooks being the double integral method \cite{Jacob}. Others methods can be found in \cite{Laplace, 2Laplace}. Conrad synthesized eleven ways for solving the Gaussian integral, among them he used the method of Fourier transforms, Stirling's formula, contour integration, and differentiation under the integral sign \cite{Conrad}. Let us now consider a more general Gaussian integral like \cite{Salinas, Hernandez,wolfram}
\begin{eqnarray} \label{fc}
\int_{-\infty}^{+\infty}x^{2n}\;e^{-\alpha x^2}\;dx &=& \dfrac{(2n-1)!!}{2^{n}}\frac{\sqrt\pi}{\alpha^{\frac{2n+1}{2}}}\;, \;\;(\forall\;n\in\mathbb{N})\;,
\end{eqnarray}
\begin{eqnarray} \label{Eq.3} 
\int_{-\infty}^{+\infty}x^{2n}\;e^{-\alpha x^2}\;dx\; &=& \dfrac{\Gamma\left(\frac{2n+1}{2}\right)}{\alpha^{\frac{2n+1}{2}}} , \;\;(\forall\; n\in\mathbb{N})\;.
\end{eqnarray}
For simplicity, we are considering that the number zero is included in the set of natural numbers. Here, $x!!$ is the double factorial and $\Gamma(x)$ is the gamma function. The gamma function has great relevance for the development of new functions that can be applied directly in physics. Normally this function is present in problems of physics such as, for example, in the normalization of Coulomb wave functions and the calculation of probabilities in statistical mechanics \cite{ARFKEN}. Notice that Eqs. (\ref{fc}) and (\ref{Eq.3}) admits a solution for $x^{2n+1}$, however, we present only the solutions for $x^{2n}$ because the method which will be presented in this work requires it.  

Inspired by the ways to solve the Gaussian integral as presented by Conrad, this work aims to present a new method for solving the Gaussian integral based on the numerical evolution of the expansion coefficients in Taylor Series of a simple function. This is relevant because enable us to treat certain mathematical and physical problems from another perspective having the differential calculus as a tool. The method consists of introducing a parameter that has the role of describing the evolution of the numerical sequence of Gaussian integral. Next, we show that this parameter can be thought as the coefficients of the  expansion of function $f(x) = (1-x)^{-1/2}$. Having done this, we were able to introduce some Theorems that are proved by mathematical induction. As a test for the method, we applied the results due to Theorems to obtain a non-extensive version for the particle number density via Tsallis statistics. Also, solutions for the gamma function of the form $\Gamma(1/2\pm n)$, and the factorial function of a kind $(n+1/2)!$, where $n\in\mathbb{N}$, are derived. Besides, we present an application possibility in fractional calculus using the definitions of fractional derivative according to Riemann-Liouville and Caputo. 

The paper is structured as follows. In Section \ref{section1}, we show the entire construction process of the new method for solving the Gaussian integral. Simple applications in Tsallis statistical framework, special functions (gamma and factorial), and fractional calculus are presented in Section \ref{section2}. Finally, the conclusions are shown in Section \ref{conclusion}.

\section{The method} \label{section1}

Consider the Gaussian integral as follows:
\begin{eqnarray} \label{igg}
I_{2n}&=&\int_{-\infty}^{+\infty}x^{2n}\;e^{-\alpha x^2}\;dx\;,\;\;(\forall\;n\in\mathbb{N}; \alpha>0)\;.
\end{eqnarray}
Solving $I_{2n}$ by conventional methods, we arrive at the following results: $I_{0}=\sqrt{\pi}/\alpha^{1/2}$, $I_{2}=\sqrt{\pi}/2\alpha^{3/2}$, $I_{4}=3\sqrt{\pi}/4\alpha^{5/2}$, $I_{6}=15\sqrt{\pi}/8\alpha^{7/2}$ and so on. We can generalize these results by introducing a parameter that depends on $n$. The new method proposed here is based on the following definition:
\begin{eqnarray} 
I_{2n} &\equiv& \gamma_{2n}\frac{\sqrt\pi}{\alpha^{\frac{2n+1}{2}}}\;,
\end{eqnarray}
where $\gamma_{2n}$ is a parameter to be determined. Note that $\gamma_{0}=1$, $\gamma_{2}=1/2$, $\gamma_{4}=3/4$, $\gamma_{6}=15/8$ and so forth. The parameter $\gamma_{2n}$ has an apparently unpredictable sequence, but it becomes evident when we perform the expansion of the function $f(x) = (1-x)^{-1/2}$ in a Taylor series around $x=0$, as bellow:
\begin{eqnarray} 
f(x) &=& \sum_{n=0}^{\infty}{\frac{d^n(1-x)^{-1/2}}{dx^n}\Biggl|}_{x=0}\frac{x^n}{n!}. 
\end{eqnarray}
Analyzing the first coefficients of the above expansion, we arrive at the following results: $f(0)=1$, $f'(0)=1/2$, $f''(0)=3/4$, $f'''(0)=15/8$, $f''''(0)=105/16$ and so on. We now intend to find an expression for the parameter $\gamma_{2n}$ from the results shown above. Note that the coefficients of the expansion generate identical numbers those generated by the evolution of the parameter $\gamma_{2n}$. Thus, it is reasonable to establish the following definition:
\begin{definition} \label{def3}
	
	For $n \in \mathbb{N}$, $\gamma_{2n}$ can be defined as 
	\begin{eqnarray} \label{gama} 
	\gamma_{2n} &\equiv& {{\frac{d^{n}(1-x)^{-1/2}}{dx^{n}}}\Biggl|}_{x=0}\;.
	\end{eqnarray}
	
\end{definition}

\begin{lemma} \label{lem1}
	
	By Definition \ref{def3} for all $n\in\mathbb{N}$, we have
	\begin{eqnarray} 
	\gamma_{2n} &=& \prod_{k=0}^{n-1}\left(\frac{1}{2} + k\right)\;. 
	\end{eqnarray}
	
\end{lemma}

\begin{proof}
	
	The proof of this result is by induction on $n$. More generally, the product notation is defined as
	\begin{eqnarray} 
	 \prod_{i=m}^{n}x_i &=& x_m\cdot x_{m+1}\cdots x_{n-1}\cdot x_n\;.
	\end{eqnarray}
	If $m>n$, the product is an empty product, whose value is $1$. The base case ($n=0$) is an empty product, hence $\gamma_{0}=1$. From Definition \ref{def3}, for now with $x\neq 0$, it is easy to see that the inductive hypothesis is
	\begin{eqnarray} \label{IH}
	\frac{d^{n}(1-x)^{-1/2}}{dx^{n}} &=& \prod_{k=0}^{n-1}\left(\frac{1}{2} + k\right)(1-x)^{-\frac{1}{2}-n}\;.
	\end{eqnarray}
	Then, we shall show that statement (\ref{IH}) it is true for $n+1$. Therefore 
	\begin{eqnarray} \label{these} \nonumber
	\frac{d^{n+1}(1-x)^{-1/2}}{dx^{n+1}} &=& \frac{d}{dx}\left[\prod_{k=0}^{n-1}\left(\frac{1}{2} + k\right)(1-x)^{-\frac{1}{2}-n}\right]\;,
	\\ 
	&=& \prod_{k=0}^{n}\left(\frac{1}{2} + k\right)(1-x)^{-\frac{1}{2}-(n+1)}\;,
	\end{eqnarray}
	where we used the inductive hypothesis to obtain the above result. Thus, at $x=0$, Eq. (\ref{these}) may take the form
	\begin{eqnarray}   
	\gamma_{2n} &=& \prod_{k=0}^{n-1}\left(\frac{1}{2} + k\right)\;,\;\;\forall\;\;n\in\mathbb{N}\;. 
	\end{eqnarray}
	 
\end{proof}

\begin{definition} \label{prop1}
	
	From Eq. (\ref{fc}) we can write
	\begin{eqnarray}
	\bar{\gamma}_{2n} &\equiv& \frac{(2n-1)!!}{2^{n}}\;,\;\;\forall\;\;n\in\mathbb{N}\;.
	\end{eqnarray}
	
\end{definition}

\noindent
The double factorial $m!!$ is defined as follow \cite{ARFKEN}:
\begin{eqnarray} 
m!! = \left\{
\begin{array}{ll}
m\cdot(m-2)\dots 5\cdot 3\cdot 1 & \;\; m>0\;\; {\rm odd} \\
m\cdot(m-2)\dots 6\cdot 4\cdot 2 & \;\; m>0\;\; {\rm even} \\
1 & \;\; m=-1,\;0\;.
\end{array}
\right.
\end{eqnarray}

\begin{theorem} \label{the1}
	
	Let $I_{2n}$ given by Eq. (\ref{fc}), then 
	\begin{eqnarray} 
	I_{2n} &=& \gamma_{2n}\;\frac{\sqrt\pi}{\alpha^{\frac{2n+1}{2}}}\;,\;\;\forall\;\;n\in\mathbb{N}\;,
	\end{eqnarray}
	where $\gamma_{2n}$ is given by Definition \ref{def3}. 
\end{theorem}

To prove the Theorem \ref{the1}, we will take into account the Definition \ref{prop1}. In other words, we just need to show that $\gamma_{2n}=\bar{\gamma}_{2n}$, for all $n\in\mathbb{N}$. 

\begin{proof}
	
	The base case ($n=0$) shows that $\gamma_{0}=\bar{\gamma}_{0}$, since from Definition \ref{prop1}, $(-1)!!=1$. Suppose that $\gamma_{2n}=\bar{\gamma}_{2n}$ it is true for $n=q$ with $q\in\mathbb{N}$, then the inductive hypothesis is given by
	\begin{eqnarray} 
	\gamma_{2q} &=& \bar{\gamma}_{2q}\;.
	\end{eqnarray}
	We now must show that $\gamma_{2n}=\bar{\gamma}_{2n}$ it is also true for $n=q+1$. So we have
	\begin{eqnarray} \label{these2}
	\gamma_{2(q+1)} &=& \bar{\gamma}_{2(q+1)}\;.
	\end{eqnarray}
	By Lemma \ref{lem1}, with $n=q+1$ and $k=i$, the left-hand side of relation (\ref{these2}) becomes 
	\begin{eqnarray} \nonumber
	\gamma_{2(q+1)} &=& \prod_{i=0}^{q}\left(\frac{1}{2} + i\right)\;, 
	\\  \nonumber
	&=& \frac{1}{2}\left( \frac{1}{2} +1\right)\cdots \left( \frac{1}{2} +(q-1)\right) \left( \frac{1}{2} + q\right)\;,
	\\ 
	&=& \gamma_{2q} \left( \frac{1}{2} + q\right)\;.           	    	
	\end{eqnarray}
	Applying the inductive hypothesis, we get
	\begin{eqnarray} \label{prove1}
	\gamma_{2(q+1)} &=& \bar{\gamma}_{2q} \left( \frac{1}{2} + q\right)\;.           	    	
	\end{eqnarray}
	Using the Definition \ref{prop1} in Eq. (\ref{prove1}), we obtain
	\begin{eqnarray} \nonumber
	\gamma_{2(q+1)} &=& \frac{(2q-1)!!}{2^{q}}\left( \frac{1}{2} + q\right)\;,
	\\ \nonumber
	&=& \frac{1}{2^{q+1}}\left[ 1\cdot 3\cdot 5 \cdots (2q-1)(2q+1)\right]\;,  
	\\ 
	&=& \bar{\gamma}_{2(q+1)}\;.          	    	
	\end{eqnarray}
	This completes the proof. 
	
\end{proof}

An immediate consequence of the Theorem \ref{the1} is that we can establish a relationship between the double factorial and the parameter $\gamma_{2n}$, as bellow:
\begin{eqnarray}
(2n-1)!! = 2^n\gamma_{2n} = 2^n\;{{\frac{d^{n}(1-x)^{-1/2}}{dx^{n}}}\Biggl|}_{x=0}\;,\;\;\forall\;\;n\in\mathbb{N}\;.
\end{eqnarray}
For example, when $n=1$, we have $1!!=1$, since $\gamma_{2}=1/2$. For $n=2$, we have $3!!=3$, since $\gamma_{4}=3/4$ and so on.

The method developed here is relevant because it presents a new form of solving the Gaussian integral of the type $I_{2n}$. Since there are other methods, this has shown more practical than some methods found in the literature. For example, in Eq. (\ref{Eq.3}), we need to solve a gamma function $\Gamma(t)$ as shows the Definition \ref{def1}, whereas our method just needs to derive a simple function. Besides, we demonstrated that parameter $\gamma_{2n}$ represents the Taylor series coefficients of the function $f(x)=(1-x)^{-1/2}$. This means that the evolution of Gaussian integral of a kind $I_{2n}$ has a strong relationship to the expansion terms of $f(x)$, a result very newsworthy. One aspect that makes this method accessible is the fact that, for example, in physical problems $n$ is not usually very large, which facilitates its application. Hereafter, we briefly present an application of our method in physics, specifically in the determination of the particle number density in Tsallis framework. Moreover, we present solutions for the gamma and factorial functions (special functions) in terms of parameter $\gamma_{2n}$, and we also show a simple application in fractional calculus.   

\section{Simple applications} \label{section2} 

\subsection{Particle number density in Tsallis framework}

We aim to present a simple application of the method proposed in this work. We chose to determine the particle number density in the Tsallis framework because an ideal scenario emerges in which it is possible to notice the functionality of the method. In 1988, Constantino Tsallis proposed a possible generalization of the Boltzmann-Gibbs (BG) entropy \cite{Tsallis}. The proposed new entropy is expressed by
\begin{eqnarray} 
S_{q} = \frac{k_{B}}{q-1}\left(1-\sum_{i=1}^{\Omega}p_{i}^{q}\right)\;,
\end{eqnarray}
where $k_B$ is the Boltzmann constant, $p_{i}$ is the probability of the system to be found in the microstate $i$ and $q$ is the parameter that characterizes the degree of nonextensivity of the system. The classical entropy is recovered in the limit $q \rightarrow 1$. In Tsallis' statistics, the particle number of species $i$ per volume can be written as follows:
\begin{eqnarray} \label{dnpeg} 
n_{i}^{q} = \frac{g_{i}}{(2\pi\hbar)^{3}}\int_{-\infty}^{+\infty} d^{3}p \mathcal{N}_{i}^{q}\;,
\end{eqnarray}
where $\mathcal{N}_{i}^{q}$ is the generalized occupation number, $g_i$ is the degeneracy of species $i$ and $\hbar$ is the Planck reduced constant. The generalized occupation number for fermions in Tsallis framework is given by \cite{Shen}
\begin{eqnarray} \label{19}
\mathcal{N}_{i}^{q} = \frac{1}{e_{2-q}^{\beta (E_{i}-\mu_{i})} + 1}\;,
\end{eqnarray}
where $e_{2-q}^{x} \equiv [1 + (q-1)x]^{1/(q-1)}$. Further, $\beta=1/k_{B}T$, $\mu_{i}$ and $E_i$ are the chemical potential and particle energy of species $i$, respectively. 

Expanding Eq. (\ref{19}) up to the first order of $(q-1)$, we obtain \cite{Pessah}
\begin{eqnarray} \label{64}
	\mathcal{N}_{i}^{q} = \frac{1}{e^{\beta (E_{i}-\mu_{i})} + 1} + \frac{(q-1)}{2}\frac{(\beta (E_{i}-\mu_{i}))^{2}e^{\beta (E_{i}-\mu_{i})}}{\left( e^{\beta (E_{i}-\mu_{i})} + 1\right)^2}\;.
\end{eqnarray}
Considering the case in which $k_{B}T<(E_{i}-\mu_{i})$, then $e^{\beta (E_{i}-\mu_{i})} \gg  1$. Thus, Eq. (\ref{64}) becomes
\begin{eqnarray} \label{Napprox}
	\mathcal{N}_{i}^{q} = e^{-\beta (E_{i}-\mu_{i})} + \frac{(q-1)}{2}(\beta (E_{i}-\mu_{i}))^{2}e^{-\beta (E_{i}-\mu_{i})}\;.
\end{eqnarray}
In this way, assuming the energy for non-relativistic particles as $E_i=m_{i}c^2+p^2/2m_i$ and using Eq. (\ref{Napprox}) in Eq. (\ref{dnpeg}), the generalized particle number density take the form
\begin{eqnarray} \label{33} \nonumber
	n_{i}^{q} &=& \frac{2\pi g_{i}}{(2\pi\hbar)^{3}}e^{\beta(\mu_{i}-m_{i}c^2)}\int_{-\infty}^{+\infty} dpp^{2} e^{-\alpha p^2}  + \frac{2\pi g_{i}}{(2\pi\hbar)^{3}}\frac{(q-1)}{2}\beta^{2}e^{\beta(\mu_{i}-m_{i}c^2)}(m_{i}c^{2}-\mu_{i})^{2}\int_{-\infty}^{+\infty} dpp^{2} e^{-\alpha p^2} + \frac{2\pi g_{i}}{(2\pi\hbar)^{3}}\frac{(q-1)}{2}\beta^{2} \\
	&\times&
	m_{i} (m_{i}c^{2}-\mu_{i})e^{\beta(\mu_{i}-m_{i}c^2)}\int_{-\infty}^{+\infty} dpp^{4} e^{-\alpha p^2} 
	+
	\frac{g_{i}}{(2\pi\hbar)^{3}}\frac{(q-1)}{2}\beta^{2}e^{\beta(\mu_{i}-m_{i}c^2)}\frac{\pi}{2m_{i}^{2}}\int_{-\infty}^{+\infty} dpp^{6} e^{-\alpha p^2}\;, 
\end{eqnarray} 
where $\alpha \equiv \beta/2m_{i}$ and $c$ is the speed of light in vacuum. 

Note that in the above expression we have four Gaussian integrals of type $I_{2n}$. It is at this point that we will apply our method. To solve these Gaussian integrals, we will use the Theorem \ref{the1}. The results are:
\begin{eqnarray} 
	\int_{-\infty}^{+\infty} dpp^{2} e^{-\alpha p^2} &= \gamma_{2}(2m_{i}k_{B}T)^{3/2}\sqrt{\pi}\;,\\ 
	\int_{-\infty}^{+\infty} dpp^{4} e^{-\alpha p^2} &= \gamma_{4}(2m_{i}k_{B}T)^{5/2}\sqrt{\pi}\;,\\ 
	\int_{-\infty}^{+\infty} dpp^{6} e^{-\alpha p^2} &= \gamma_{6}(2m_{i}k_{B}T)^{7/2}\sqrt{\pi}\;.
\end{eqnarray}
We will use the Definition \ref{def3} to compute the coefficients $\gamma_{2n}$, thus we obtain $\gamma_{2}=1/2$, $\gamma_{4}=3/4$ and $\gamma_{6}=15/8$. Finally, we can find the generalized particle number density. After a little algebra, we get \cite{Pessah}
\begin{eqnarray} 
	n_{i}^{q} &= g_{i}\left(\frac{m_{i}k_{B}T}{2\pi\hbar^{2}}\right)^{3/2} e^{\frac{\mu_{i}-m_{i}c^{2}}{k_{B}T}}\left\lbrace  1 + \frac{(q-1)}{2}\left[ \left(\dfrac{m_{i}c^{2}-\mu_{i}}{k_{B}T}\right)^{2}  + 3\left(\dfrac{m_{i}c^{2}-\mu_{i}}{k_{B}T}\right) + \dfrac{15}{4} \right] \right\rbrace\;.
\end{eqnarray}
Notice that the usual particle number density is recovered when $q=1$.

We chose to apply our method to determine the particle number density in Tsallis framework because it arises integrals in which the $n$ index assumes three different values that facilitated the application.  This enables us to visualize from another perspective how the Gaussian integral evolve as $n$ grows. As mentioned before, this method is viable to apply in mathematical and physical problems where the $n$ index is small, since we would need to perform derivatives successive to obtain the result. On the other hand, the derivative order is half of the value of $2n$ index, for example, in Eq. (\ref{33}) the integral containing the term $p^6$, we need to derive only three times, making the practical and efficient method in the resolution of integrals of type $I_{2n}$. Even that $n$ is large, it was demonstrated that the method also works for all $n\in \mathbb{N}$. 

\subsection{Special functions}

An immediate application of Theorem \ref{the1} is verified in the gamma and factorial functions. Below is the definition of the special functions we will consider from now on. 

\begin{definition} \label{def1}
	
	(Euler, 1730) Let $t \in \mathbb{R}$, and $t>0$, the gamma function is defined by \cite{BOAS,Riley}
	\begin{eqnarray} \label{FG}
	\Gamma(t) &=& \int_{0}^{+\infty}x^{t-1}\;e^{- x}\;dx\;.
	\end{eqnarray}
		
\end{definition}

\noindent
Using $t=p+1$, and integrating by parts, we obtain the following recurrence relation: 
\begin{eqnarray} \label{2.12}
\Gamma(p+1) &=& p\Gamma(p)\;.
\end{eqnarray}

\begin{definition} \label{def2}
	
	For $m>-1$, the factorial function is defined by \cite{BOAS,Riley}
	\begin{eqnarray} \label{3.21}
	m! &=& \int_{0}^{+\infty}x^{m}\;e^{-x}\;dx\;.
	\end{eqnarray}
	
\end{definition}

We begin with the gamma function given by Definition \ref{def1}, setting $t=n+1/2$, that is
\begin{eqnarray} 
\Gamma\left(\frac{2n+1}{2}\right)  &=& \int_{0}^{+\infty}x^{\frac{2n-1}{2}}\;e^{-x}\;dx\;. 
\end{eqnarray}
Taking $x=r^{2}$, 
\begin{eqnarray} 
\Gamma\left(\frac{2n+1}{2}\right) &=& \int_{-\infty}^{+\infty}r^{2n}\;e^{-r^2}\;dr\;. 
\end{eqnarray}
Note that we can apply Theorem \ref{the1}, with $\alpha=1$, in the above expression to obtain 
\begin{eqnarray} \label{gammafunction}
\Gamma\left(\frac{2n+1}{2}\right) &=& \gamma_{2n}\sqrt{\pi}\;. 
\end{eqnarray}
Having presented (\ref{gammafunction}), it is convenient to introduce a recurrence relation for the parameter $\gamma_{2n}$.   

\begin{lemma} \label{lem2}
	
	If $\gamma_{2n}$ is given by Definition \ref{def3}, then
	\begin{eqnarray} 
	\gamma_{2(n+1)} &=& \frac{2n+1}{2}\gamma_{2n}\;.
	\end{eqnarray}
	
\end{lemma} 

\begin{proof}
	
	Using (\ref{gammafunction}), and the succeeding term, given by
	\begin{eqnarray} 
	\Gamma\left(\frac{2n+1}{2}+1\right) &=& \gamma_{2(n+1)}\sqrt{\pi}\;,
	\end{eqnarray} 
	we can insert these expressions in Eq. (\ref{2.12}), defining $p=n+1/2$. Therefore
	\begin{eqnarray} 
	\gamma_{2(n+1)} &=& \frac{2n+1}{2}\gamma_{2n}\;. 
	\end{eqnarray}
	
\end{proof}      

The Theorem \ref{the1} ensure us that the parameter $\gamma_{2n}$ is given by Definition \ref{def3}. Based on this, it is straightforward to show that the following equality is valid:
\begin{eqnarray} 
\sum_{n=0}^{\infty}\gamma_{2n}\frac{x^n}{n!} &=& \sum_{n=0}^{\infty}(-1)^{n}x^{n} \binom{-\frac{1}{2}}{n}\;, 
\end{eqnarray}
where $\binom{x}{y}$ is the binomial coefficient. From the above expression, we can conclude that 
\begin{eqnarray}
	\gamma_{2n} &=& (-1)^{n}\frac{\sqrt{\pi}}{\left(-\frac{1}{2}-n\right)!}\;,
\end{eqnarray}
which yields the following result:
\begin{eqnarray}
\Gamma\left(\frac{1}{2}-n\right)  &=& (-1)^{n}\frac{\sqrt{\pi}}{\gamma_{2n}}\;,
\end{eqnarray}
where we use the identity $m!=\Gamma(m+1)$.

We now let us apply Theorem \ref{the1} in the factorial function given by Definition \ref{def2}, putting $m=n+1/2$. Then we have 
\begin{eqnarray} 
\left(\frac{2n+1}{2}\right)! &=& \int_{0}^{+\infty}x^{\frac{2n+1}{2}}\;e^{-x}\;dx\;. 
\end{eqnarray}
Replacing again $x=r^{2}$, we find
\begin{eqnarray} 
\left(\frac{2n+1}{2}\right)! &=& \int_{-\infty}^{+\infty}r^{2(n+1)}\;e^{-r^2}\;dr\;. 
\end{eqnarray}
Hence using Theorem \ref{the1}, we obtain
\begin{eqnarray} 
\left(\frac{2n+1}{2}\right)! &=& \gamma_{2(n+1)}\sqrt{\pi}\;, 
\end{eqnarray} 
and making use of Lemma \ref{lem2}, we get
\begin{eqnarray} \label{46}
\left(\frac{2n+1}{2}\right)! &=& \frac{2n+1}{2}\gamma_{2n}\sqrt{\pi}\;. 
\end{eqnarray}

\subsection{Fractional derivative: a possibility}

Here we intend to present an interesting result that may be useful in fractional calculus. Initially let us consider the following definitions: 

\begin{definition} \label{hyperf}
	
	Let $Re(c)>Re(b)>0$, the hypergeometric function is defined by \cite{Weisstein}
	\begin{eqnarray}
	{_{2}F_{1}}\left(a,b;c;z\right) &=& \frac{\Gamma(c)}{\Gamma(b)\Gamma(c-b)}\int_{0}^{1}\tau^{b-1}(1-\tau)^{c-b-1}(1-z\tau)^{-a}d\tau\;,
	\end{eqnarray}
	where $Re(x)$ is real part of $x$.
\end{definition}

\begin{definition} \label{defRL}
	
	The fractional derivative according to Riemann-Liouville is defined by \cite{Oliveira}
	\begin{eqnarray}
		D^{\alpha}f(t) &=& \frac{1}{\Gamma(m-\alpha)}\frac{d^{m}}{dt^m}\int_{0}^{t}\frac{f(\tau)}{(t-\tau)^{\alpha-m+1}}d\tau\;,
	\end{eqnarray}
	where $\alpha$ is a complex number such that $Re(\alpha)>0$ and yet shall is in interval $m-1<Re(\alpha)\leq m$ with $m \in \mathbb{N}$.  
	
\end{definition}

\begin{definition} \label{defCaputo}
	
	The fractional derivative according to Caputo is given by \cite{Oliveira}
	\begin{eqnarray}
	D^{\alpha}f(t) &=& \frac{1}{\Gamma(m-\alpha)}\int_{0}^{t}(t-\tau)^{m-\alpha-1}\frac{d^{m}}{d\tau^m}f(\tau)d\tau\;,
	\end{eqnarray}
	being $Re(\alpha)>0$ such that $m-1<Re(\alpha)\leq m$ with $m \in \mathbb{N}$.
	
\end{definition}

By using $f(t)=t^\theta$ with $\theta>-1$ in Definition \ref{defRL}, we obtain
\begin{eqnarray}
	D^{\alpha}t^\theta &=& \frac{\Gamma(\theta+1)}{\Gamma(\theta-\alpha+1)}t^{\theta-\alpha}\;.
\end{eqnarray}
For the case where $\alpha=\theta$, we have
\begin{eqnarray}
\frac{d^{\alpha}t^\alpha}{dt^\alpha} &=& \alpha!\;.
\end{eqnarray}
Let us now show that the above result can be represented by the $k$-th derivative of a simple function applied to a point.

\begin{lemma} \label{lem4}
	
	Let $k \in \mathbb{N}$, then
	\begin{eqnarray} 
	{{\frac{d^{k}(1-x)^{-1}}{dx^{k}}}\Biggl|}_{x=0} &=& k!\;.
	\end{eqnarray}
\end{lemma} 

\begin{proof}
	
	In fact,
	\begin{eqnarray} \nonumber
		\frac{d^{k}(1-x)^{-1}}{dx^{k}} &=& \Gamma(k+1)\sum_{n=k}^{\infty}\binom{n}{k}x^{n-k}\;, \\
		&=& \frac{k!}{(1-x)^{k+1}}\;, 
	\end{eqnarray}
	where 
	\begin{eqnarray}
		\sum_{n=k}^{\infty}\binom{n}{k}x^{n-k} &=& \frac{1}{(1-x)^{k+1}}\;.
	\end{eqnarray}
	Hence, at $x=0$, we obtain
	\begin{eqnarray} 
	{{\frac{d^{k}(1-x)^{-1}}{dx^{k}}}\Biggl|}_{x=0} &=& k!\;.
	\end{eqnarray}
	
\end{proof}

Suppose that $k$ can take the form $k=n + 1/2$, the Lemma \ref{lem4} may be written as follow:
\begin{eqnarray} \label{55}
{{\frac{d^{\frac{2n+1}{2}}(1-x)^{-1}}{dx^{\frac{2n+1}{2}}}}\Biggl|}_{x=0} &=& \left(\frac{2n+1}{2}\right)!\;.
\end{eqnarray}
Let us now analyze the left-hand side of the above equation. To accomplish this, we will use the Definitions \ref{defRL} and \ref{defCaputo} to evaluate the fractional derivatives. Thus, using the Definition \ref{defRL} with $\alpha=n+1/2$, we find the following result: 
\begin{eqnarray}
	D^{\frac{2n+1}{2}}(1-x)^{-1} &=& \frac{x^{-n-\frac{1}{2}}}{\Gamma\left(\frac{1}{2}-n\right)}\; {_{2}F_{1}}\left(1,1;m-n+\frac{1}{2};x\right)\;, 
\end{eqnarray}
where ${_{2}F_{1}}\left(a,b;c;z\right)$ is the hypergeometric function given by the Definition \ref{hyperf}. This result applied at the point $x=0$ diverge.

Now using the Definition \ref{defCaputo} with $\alpha=n+1/2$ on the left-hand side of Eq. (\ref{55}), we obtain
\begin{eqnarray}
D^{\frac{2n+1}{2}}(1-x)^{-1} &=& \frac{\Gamma(m+1)}{\Gamma\left(m-n+\frac{1}{2}\right)}x^{m-n-\frac{1}{2}}\; {_{2}F_{1}}\left(m+1,1;m-n+\frac{1}{2};x\right)\;. 
\end{eqnarray}
Note that at the point $x=0$, the above result goes to zero, since $m>n+1/2$. We show that both Riemann-Liouville and Caputo definitions are flimsy when computed the fractional derivatives of the function $f(x)=(1-x)^{-1}$ at the point $x=0$. On the other hand, this problem can be avoided considering the result presented in Eq. (\ref{46}). Hence, we may write 
\begin{eqnarray} \label{Eq.60}
{{D^{\frac{2n+1}{2}}(1-x)^{-1}}\bigl|}_{x=0} &=& \frac{2n+1}{2}\gamma_{2n}\sqrt{\pi}\;.
\end{eqnarray}

\section{Conclusions} \label{conclusion}

It was presented a new method for solving the Gaussian integral inspired by expansion in Taylor series of a simple function, namely $f(x)=(1-x)^{-1/2}$. Introducing a parameter with a $n$ index dependence, we have found a general solution for this type of integral being able to use in any situation, since the Gaussian integral is of type $I_{2n}$. We have demonstrated that the parameter $\gamma_{2n}$ represents the Taylor series coefficients of the function $f(x)$, a result newsworthy. The reliability of the method is guaranteed through the proof of some Theorems, which were proved by mathematical induction. To check the functionality of the method, we have obtained the particle number density in the Tsallis framework. Besides, we have presented solutions for the gamma function of the form $\Gamma(1/2\pm n)$, and the factorial function of a kind $(n+1/2)!$, in terms of the parameter $\gamma_{2n}$. We have also shown that the method is useful in fractional calculus, for example, the definitions of fractional derivative according to Riemann-Liouville and Caputo are flimsy when evaluated the fractional derivatives of the function $f(x)=(1-x)^{-1}$ at the point $x=0$, whereas using our method we find the result as shows Eq. (\ref{Eq.60}), showing again the functionality and efficiency of the method. In conclusion, we believe in the relevance of the present work because it reveals a new way of solving the Gaussian integral: one of the most famous integrals of the exact sciences. 

\section*{Acknowledgements}

The authors are very grateful to R. C. Duarte, N. S. Almeida, T. Dumelow and A. M. Filho by the helpful discussions and the Brazilian agency CAPES for financial support.


\begin{thebibliography}{20}
	
\bibitem{Davis} P. J. Davis, The American Mathematical Monthly \textbf{66}, 847-869 (1959).

\bibitem{Gronau} D. Gronau, Teaching Mathematics and Cumputer Science \textbf{1}, 43-53 (2003).

\bibitem{Stahl} S. Stahl, Mathematics Magazine \textbf{79}, 2, 96-113 (2006).

\bibitem{Murray} M. R. Spiegel, J. Schiller, and R. A. Srinivasan, {\it Probability And Statistics} (Mc Graw Hill, 2001).

\bibitem{QGreiner} W. Greiner, {\it Quantum Mechanics An Introduction} (Springer, Berlin, 1990).  

\bibitem{Sakurai} J.J. Sakurai, {\it Modern Quantum Mechanics} (AddisonWesley, Redwood City 1985). 

\bibitem{Greiner} W. Greiner, {\it Thermodynamics and Statistical Mechanics} (Springer Verlag, New York, 1995).

\bibitem{Pathria} R. K. Pathria, {\it Statistical Mechanics} (Butterworth-Heinemann, Oxford, 1996). 

\bibitem{Salinas} S. R. A. Salinas, {\it Introduction to Statistical Physics} (Springer, New York, 2001).

\bibitem{Jacob} J. K. F. Sturm, {\it Cours d'Analyse de l’\'{e}cole polytechnique} (Paris, Mallet-Bachelier, 1857).

\bibitem{Laplace} S. M. Stigler, Statistical Science \textbf{1}, 359–378, (1986).

\bibitem{2Laplace} P. S. Laplace, {\it Th\'{e}orie Analytiques des Probabilit\'{e}s}. Paris, Courcier 1812, reprinted (Bruxelles, Culture et Civilisation, 1967).

\bibitem{Conrad}
K. T. Conrad, The Gaussian Integral. (2013). Online at \url{https://www.semanticscholar.org/paper/THE-GAUSSIAN-INTEGRAL-Conrad/4687538f80e333c175691d627dc1254eef3605f8}

\bibitem{Hernandez} S. M. Hernandez, {\it Termodin\`{a}mica i Mec\`{a}nica estad\'{i}stica} (Libre Lulu, Breda, 2015).

\bibitem{wolfram}
E. W. Weisstein, {\it Gaussian integral}. From MathWorld-A Wolfram Web Resource. \url{http://mathworld.wolfram.com/GaussianIntegral.html}

\bibitem{ARFKEN} G. B. Arfken and H. J. Weber, {\it Mathematical Methods for Physicists} (Academic Press, New York, 2005).

\bibitem{Tsallis} 
C. Tsallis, \textit{Possible generalization of Boltzmann-Gibbs statistics} {\bf 52}, 479-487 (1988).

\bibitem{Shen}
K.-M. Shen, B.-W. Zhang, and E.-K. Wang, Physica A: Statistical Mechanics and its Applications {\bf 487}, 215 (2017)

\bibitem{Pessah}
M. E. Pessah, D. F. Torres and H. Vucetich, {\it Statistical mechanics and the description of the early universe. (I). Foundations for a slightly non-extensive cosmology} {\bf 297}, 164-200 (2001).	

\bibitem{BOAS} M. L. Boas, {\it Mathematical methods in the physical sciences}. (John Wiley \& Sons, New Jersey, 2006).

\bibitem{Riley} K. F. Riley, M. P. Hobson, and S. J. Bence, {\it Mathematical Methods for Physics and Engineering} (Cambridge University Press, Cambridge, 2006).

\bibitem{Weisstein}
E. W. Weisstein, {\it Hypergeometric Function}. From MathWorld-A Wolfram Web Resource. \url{https://mathworld.wolfram.com/HypergeometricFunction.html}

\bibitem{Oliveira}
E. C. Oliveira and J. A. T. Machado, {\it A review of definitions for fractional derivatives and integral}, Mathematical Problems in Engineering {\bf 2014}, (2014).
	
\end{thebibliography}
\end{document}